\documentclass[a4paper,10pt]{amsart}

\usepackage{acatta}
\newcommand{\blankSpace}{\cdot}

\newcommand{\LC}{\nabla^{LC}}
\newcommand{\ccon}{\nabla^c}
\newcommand{\scal}{\operatorname{scal}}
\newcommand{\Ric}{\operatorname{Ric}}
\newcommand{\const}{\text{const}}

\title[Kodaira dimension of almost K\"ahler manifolds and curvature]{Kodaira dimension of almost K\"ahler manifolds and curvature of the canonical connection}

\author{Andrea Cattaneo}

\author{Antonella Nannicini}
\address{Dipartimento di Matematica ed Informatica ``U. Dini''\\
Universit\`a degli Studi di Firenze\\
Viale Morgagni 67/A\\
50134 Firenze, Italy}
\email{andrea.cattaneo@unifi.it}
\email{antonella.nannicini@unifi.it}

\author{Adriano Tomassini}
\address{Dipartimento di Scienze Matematiche, Fisiche ed Informatiche
Unit\`a di Matematica e Informatica\\
Universit\`a di Parma\\
Parco Area delle Scienze 53/A, 43124\\
Parma, Italy}
\email{adriano.tomassini@unipr.it}

\keywords{Kodaira dimension, almost complex manifolds, almost K\"ahler manifolds, canonical connection.}
\thanks{This work was partially supported by the Project PRIN ``Varietà reali e complesse: geometria, topologia e analisi armonica'' and by GNSAGA of INdAM. Andrea Cattaneo is also supported by the `Grant de Bartolomeis', a fellowship in memory of Prof. Paolo de Bartolomeis.}
\subjclass[2010]{53C55, 53C25}

\begin{document}

\begin{abstract}
The notion of Kodaira dimension has recently been extended to general almost complex manifolds. In this paper we focus on the Kodaira dimension of almost K\"ahler manifolds, providing an explicit computation for a family of almost K\"ahler threefolds on the differentiable manifold underlying a Nakamura manifold. We concentrate also on the link between Kodaira dimension and the curvature of the canonical connection of an almost K\"ahler manifold, and show that in the previous example (and in another one obtained from a Kodaira surface) the Ricci curvature of the almost K\"ahler metric vanishes for all the members of the family.
\end{abstract}

\maketitle

\tableofcontents

\section{Introduction}

When studying complex manifolds, one of the first invariants one can attach to any given complex manifold is its Kodaira dimension. This invariant captures the geometry of the manifold $X$ under consideration expressing the rate of growth of the plurigenera $P_m(X) = \dim_\IC H^0(X, \cK_X^{\tensor m})$ with respect to $m$. The definition of the Kodaira dimension has recently been extended by Chen--Zhang in the setting of almost complex manifolds (see \cite{Chen-Zhang}). Among the main points addressed in this paper, there are two which, according to us, deserve particular attention: first of all the proof that the the spaces of sections of the pluricanonical bundles $\cK_X^{\tensor m}$ are finite dimensional, and then the attention one must pay to properly define what a pseudoholomorphic pluricanonical section is. Regarding these points, up to now the state of the art does not provide tools for the actual computations of the spaces of pluricanonical sections other than the definitions, which makes the determination of the Kodaira dimension of an almost complex manifold extremely challenging.

The aim of the present note is to show some of the features of this extended version of the Kodaira dimension, focussing in particular in the case of (non-integrable) almost K\"ahler manifolds. We present some results in complex dimension $2$ and $3$: we endow the differentiable manifolds underlying a Kodaira--Thurston surface and a completely solvable Nakamura threefold with families of almost complex structures and Riemannian metrics turning them into families of almost K\"ahler manifolds. In particular, we prove the following

\begin{thm*}[{Theorem \ref{thm: kod dim}}]
There exist a family of almost complex structures $J_t$ with $t = (t_1, t_2, t_3, t_4) \in \IR^4$ on the differentiable manifold $N$ underlying the Nakamura threefold such that
\[\kappa^{J_t}(N) = \left\{ \begin{array}{ll}
0 & \text{if } t_4 = 0,\\
-\infty & \text{if } t_4 \neq 0.
\end{array} \right.\]
\end{thm*}

It is known that almost Hermitian manifolds carry a canonical connection on their tangent bundle (in the integrable case, it is the Chern connection). Our second aim is to study the relationship between the curvature of the canonical connection and the Kodaira dimension. In the integrable case, a theorem of Yau (see \cite[Corollary 2]{Yau}) states that on a compact K\"ahler manifold the positivity of the total scalar curvature of the Chern connection forces the Kodaira dimension of the manifold to be $-\infty$; a generalization of this result for almost Hermitian manifolds is provided in \cite[Theorem 1.1]{Yang2}, \cite[Theorem 1.3]{Yang} and \cite[Proposition 9.5]{Chen-Zhang}. Our results show that the opposite implication does not hold in general: by computing explicitly the scalar curvature of the canonical connection of our examples, we find that it is possible for an almost K\"ahlelr manifold to have vanishing scalar curvature and Kodaira dimension $0$. More precisely, we prove the following

\begin{thm*}[{Theorem \ref{thm: scal kodaira} and \ref{thm: scal nakamura}}]
There exist families $X_a$ and $Y_t$ of almost K\"ahler manifolds (with $a \in \IR \smallsetminus \set{0}$ and $t \in \IR^4$) whose members have Kodaira dimension $-\infty$ on a dense subset of the parameter space, and whose canonical connection $\ccon$ has $\Ric(\ccon) \equiv 0$ (hence also $\scal(\ccon) \equiv 0$).
\end{thm*}

A final outcome of our work can be obtained by combining the previous two results. As we mentioned, the different members of the families we consider have different Kodaira dimension and vanishing scalar curvature. More in detail, all the members have Kodaira dimension $-\infty$ except those on a subvariety of the parameter space where the Kodaira dimension jumps to $0$; on the other hand, for all the members of these families the reason why the scalar curvature vanishes is that the canonical connection has trivial Ricci tensor. Hence we show also that in the almost K\"ahler case it is possible for a manifold to have vanishing Ricci curvature (hence trivial first Chern class) but Kodaira dimension $-\infty$.

The structure of the paper is as follows. In Section \ref{sect: kod dim general} we recall the definition of Kodaira dimension for almost complex manifolds from \cite{Chen-Zhang}. In Section \ref{sect: recaps} we collect some known results concerning the canonical connection on an almost complex manifold and its Ricci and scalar curvature, focussing in particular on the case of almost K\"ahler manifolds. In Section \ref{sect: kodaira manifold} we compute the curvature of the canonical connection on a family of almost Kahler structures on the family of almost complex manifolds introduced in \cite[$\S$6.1]{Chen-Zhang} on the differentiable manifold underlying a Kodaira surface, showing our first main result (Theorem \ref{thm: scal kodaira}). In Section \ref{sect: nakamura threefold} we consider the differentiable manifold underlying a Nakamura threefold, and endow it with a family of almost K\"ahler structures: in Section \ref{sect: kod dim of nakamura def} we compute the Kodaira dimension of these almost complex manifolds, and prove that it can assume the values $0$ or $-\infty$ (Theorem \ref{thm: kod dim}); finally in Section \ref{sect: curv of nakamura def} we show that the Ricci and scalar curvature of the almost K\"ahler metrics on the member of this family always vanish (Theorem \ref{thm: scal nakamura}).

\begin{ack}
The authors express their gratitude to Weiyi Zhang for having introduced them to the subject of Kodaira dimension for almost complex manifolds. We also thank Tian-Jun Li for having brought to our attention the reference \cite{Li} and Valentino Tosatti for his comments on a previous version of this paper.
\end{ack}

\section{Kodaira dimension of almost complex manifolds}\label{sect: kod dim general}

Let $(M, J)$ be a compact $2n$-dimensional smooth manifold endowed with an almost complex structure $J$. Following \cite{Chen-Zhang}, we recall briefly the definition of \emph{Kodaira dimension} of $(M, J)$.

Let $\Lambda_J^{p, q} M$ be the bundle of $(p, q)$-forms on $(M, J)$ and denote by $\Omega_J^{p, q}(M) = \Gamma(M, \Lambda_J^{p, q} M)$ the space of $(p, q)$-forms on $(M, J)$. Denote by $d$ the exterior differential, then
\[d(\Omega_J^{p, q}(M)) \subset \Omega_J^{p + 2, q - 1}(M) + \Omega_J^{p + 1, q}(M) + \Omega_J^{p, q + 1}(M) + \Omega_J^{p - 1, q + 2}(M).\]
Consequently $d$ splits as
\[d = A_J + \del_J + \delbar_J + \overline{A}_J,\]
where $A_J = \pi^{p + 2, q - 1} \circ d$, $\delbar_J = \pi^{p, q + 1} \circ d$. Let $\cK_X = \Lambda_J^{n, 0} M$ be the \emph{canonical bundle} of the almost complex manifold $X = (M, J)$. Then $\cK_X$ is a complex line bundle over $X$ and the $\delbar_J$-operator on $(M, J)$ gives rise to a pseudoholomorphic structure on $\cK_X$, i.e., a differential operator still denoted by $\delbar_J$,
\[\delbar_J: \Gamma(M, \cK_X) \to \Gamma(M, T^*M^{0,1} \tensor \cK_X)\]
satisfying the Leibniz rule
\[\delbar_J (f \sigma) = \delbar_J f \tensor \sigma + f \delbar_J \sigma,\]
for every smooth function $f$ and section $\sigma$.

By Hodge Theory (see \cite[Theorem 1.1]{Chen-Zhang}), $H^0(M,\cK_X^{\tensor m})$ is a finite dimensional complex vector space for every $m \geq 1$.

\begin{defin}[{\cite[Definition 1.2]{Chen-Zhang}}]
The \emph{$m^{\text{th}}$-plurigenus} of $(M, J)$ is defined as
\begin{equation}\label{plurigenus}
P_m(M, J) := \dim_\IC H^0(M,\cK_X^{\otimes m}).
\end{equation}
The \emph{Kodaira dimension} of $(M, J)$ is defined as
\begin{equation}\label{kodaira-dimension}
\kappa^J(M) := \left\{
\begin{array}{ll}
-\infty & \text{if } P_m(J) = 0 \text{ for every } m \geq 1,\\
\displaystyle \limsup_{m \to +\infty} \frac{\log P_m(J)}{\log m} & \text{otherwise}.
\end{array}
\right.
\end{equation}
\end{defin}

In their paper, Chen and Zhang provide also another definition of Kodaira dimension for an almost complex manifold (see \cite[Definition 1.5]{Chen-Zhang}): one uses a basis for the space of pseudoholomorphic sections of the pluricanonical bundle to produce a map
\[\Phi_{\cK_X^{\tensor m}}: X \smallsetminus B \longrightarrow \IP^n,\]
where $B$ is the base locus of $|\cK_X|$, and then define
\[\kappa_J(M) := \left\{ \begin{array}{ll}
-\infty & \text{if } P_m(J) = 0 \text{ for every } m \geq 1,\\
\max_m \dim_\IC \Phi_{\cK_X^{\tensor m}}(X \smallsetminus B) & \text{otherwise}.
\end{array} \right.\]

\begin{rem}
It is an open problem whether these two definitions actually coincide, but there are few cases where this is known. By definition, for an almost complex manifold $(M, J)$ we have $\kappa^J(M) = -\infty$ if and only if $\kappa_J(M) = -\infty$. It requires some moments more of thinking the fact that also $\kappa^J(M) = 0$ if and only if $\kappa_J(M) = 0$. Anyway, it is a well-known fact that $\kappa^J(M) = \kappa_J(M)$ if $J$ is integrable.
\end{rem}

\section{Recaps on the canonical connection on almost complex manifolds}\label{sect: recaps}

In this section we recall some basic facts and definitions concerning canonical connections on almost complex manifolds. The theory is well known, but we decided to include this section for the sake of completeness and to set up the notation we will use throughout the paper.

The interested reader may refer to \cite{Gau} or \cite{Tos-Wein-Yau} for a more detailed exposition.

\subsection{Generalities on connections}

We begin recalling the definition of complex connection.

\begin{defin}[Complex connection]
Let $M$ be a smooth manifold, and let $E$ be a complex vector bundle on $M$. A \emph{(complex) connection} on $E$ is a map
\[\nabla: \Gamma(M, T_\IC M) \times \Gamma(M, E) \longrightarrow \Gamma(M, E)\]
such that:
\begin{enumerate}
\item $\nabla$ is $\IC$-linear in each entry;
\item $\nabla_{fX} s = f \nabla_X s$ for every (complex smooth) function $f$ on $M$;
\item $\nabla_X (fs) = X(f) \cdot s + \nabla_X s$ for every (complex smooth) function $f$ on $M$.
\end{enumerate}
\end{defin}

If we have a real vector bundle $E$ on the manifold $M$, endowed with a (real) connection $D$, then there is a canonical way to extend this connection to a complex connection $D^\IC$ on the complexification $E_\IC$ of $E$:
\[D^\IC_{X + \ii Y} (s + \ii t) = D_X s - D_Y t + \ii (D_X t + D_Y s).\]

Let now consider a complex vector bundle on $M$. We can see our complex vector bundle as a pair $(E, I)$, where $E$ is a real vector bundle on $M$ and $I$ is an endomorphism of $E$ such that $I^2 = -\id_E$ (cf. \cite[Definition 1.1]{deBar-Tian}). For this reason, we will refer to $(E, I)$ as the complex vector bundle, while $E$ will denote the underlying real bundle. Of course, there is a canonical isomorphism of complex vector bundles $(E, I) \simeq E^{1, 0} \subseteq E_\IC$.

Let $D$ be a (real) connection on $E$. We define
\[\begin{array}{rccc}
\nabla^D: & \Gamma(M, T_\IC M) \times \Gamma(M, (E, I)) & \longrightarrow & \Gamma(M, (E, I))\\
 & (X + \ii Y, s) & \longmapsto & D_X s + I D_Y s.
\end{array}\]
The following Lemma is well known.

\begin{lemma}\label{lemma: iso connections}
In the above situation, $\nabla^D$ is a (complex) connection on $(E, I)$ if and only if $DI = 0$. In this case, $\nabla^D$ coincides with the restriction of $D^\IC$ to $E^{1, 0}$ under the canonical isomorphism
\[\begin{array}{rccc}
\xi: & (E, I) & \longrightarrow & E^{1, 0}\\
 & s & \longmapsto & \frac{1}{2} (s - \ii Is).
\end{array}\]
\end{lemma}
%

Lemma \ref{lemma: iso connections} essentially states that if $(E, I)$ is a complex vector bundle and $D$ is a connection of $E$ such that $DI = 0$, then we have a commutative diagram
\[\xymatrix{\Gamma(M, T_\IC M) \times \Gamma(M, (E, I)) \ar[r]^(.65){\nabla^D} \ar[d]_{\id \times \xi} & \Gamma(M, (E, I)) \ar[d]^\xi\\
\Gamma(M, T_\IC M) \times \Gamma(M, E^{1, 0}) \ar[r]^(.65){D^\IC} & \Gamma(M, E^{1, 0})}\]
where the vertical maps are isomorphisms. As a consequence, we have canonical bijections between the following sets:
\begin{enumerate}
\item $\set{\text{Real connections } D \text{ on } E \text{ such that } DI = 0}$;
\item $\set{\text{Complex connections on } (E, I)}$;
\item $\set{\text{Complex connections on } E^{1, 0}}$.
\end{enumerate}

\subsection{The type of a form with values in a bundle}

In this section we want to discuss some classical stuff on the type decomposition on almost complex manifolds. We restrict ourselves to the case of $2$-forms as this is the only case we will consider in the sequel.

Let $(E, I)$ be a complex vector bundle on the almost complex manifold $(M, J)$. From the real point of view, a $2$-form on $M$ with values in $E$ is a section
\[\omega \in \Gamma\pa{M, {\bigwedge}^2 T^* M \tensor_\IR E}.\]
When we extend this form by $\IC$-linearity, we get then a section
\[\hat{\omega} \in \Gamma\pa{M, {\bigwedge}^2 T^*_\IC M \tensor_\IC (E, I)}.\]
It makes then sense to decompose
\[\hat{\omega} = \omega^{2, 0} + \omega^{1, 1} + \omega^{0, 2}\]
according to the type decomposition \emph{of the form part} of $\hat{\omega}$. The relation between $\omega^{p, q}$ and $\omega$ is outlined in the following Lemma.

\begin{lemma}
Keep the notations as above. Then:
\begin{enumerate}
\item\label{item: type 20} the form $\hat{\omega}$ is of pure type $(2, 0)$ if and only if $\omega(JX, Y) = I \omega(X, Y)$;
\item the form $\hat{\omega}$ is of pure type $(1, 1)$ if and only if $\omega(JX, JY) = \omega(X, Y)$;
\item the form $\hat{\omega}$ is of pure type $(0, 2)$ if and only if $\omega(JX, Y) = -I \omega(X, Y)$.
\end{enumerate}
\end{lemma}
\begin{proof}
As the proof of each point is very similar (and these points should also be familiar), we give a proof only of \eqref{item: type 20}.

Let $X, Y \in \Gamma(M, T_\IC M)$, and denote $X^{1, 0}$ (resp., $X^{0, 1}$) the $(1, 0)$-part (resp., the $(0, 1)$-part) of $X$, and similarly for $Y$. Then $\hat{\omega}$ is of pure type $(2, 0)$ if and only if
\[\hat{\omega}(X, Y) = \hat{\omega}(X^{1, 0}, Y^{1, 0}).\]
Assume this holds, and let $X, Y \in \Gamma(M, T M)$. Then $\omega(X, Y) = \hat{\omega}(X, Y)$, and so
\[\begin{array}{rl}
\omega(X, Y) = & \hat{\omega}(X^{1, 0}, Y^{1, 0}) =\\
= & \frac{1}{4}(\hat{\omega}(X - \ii JX, Y - \ii JY)) =\\
= & \frac{1}{4}(\omega(X, Y) - \omega(JX, JY) - I(\omega(X, JY) + \omega(X, JY))).
\end{array}\]
A similar computation shows that
\[\omega(JX, Y) = \frac{1}{4}(\omega(X, JY) + \omega(X, JY) + I(\omega(X, Y) - \omega(JX, JY))),\]
hence that $\omega(JX, Y) = I \omega(X, Y)$.

Vice versa, observe that $\omega(JX, Y) = I \omega(X, Y)$ implies that also $\omega(X, JY) = I \omega(X, Y)$. It then follows that
\[\begin{array}{rl}
\hat{\omega}(X^{1, 0}, Y^{1, 0}) = & \frac{1}{4}(\hat{\omega}(X - \ii JX, Y - \ii JY)) =\\
= & \frac{1}{4}(\hat{\omega}(X, Y) - \hat{\omega}(JX, JY) +\\ 
 & -I(\hat{\omega}(X, JY) + \hat{\omega}(X, JY))) =\\
= & \frac{1}{4}(\hat{\omega}(X, Y) + \hat{\omega}(X, Y) + \hat{\omega}(X, Y) + \hat{\omega}(X, Y))) =\\
= & \hat{\omega}(X, Y).
\end{array}\]
\end{proof}

This Lemma justifies the definition of type of a form with values in a complex bundle given in \cite[Definition 1]{Gau}. Here we provide the complex interpretation, comparing $\hat{\omega}$ with the `usual' complex extension
\[\omega_\IC \in \Gamma\pa{M, {\bigwedge}^2 T^*\IC M \tensor_\IC E_\IC}\]
of $\omega$. It is in fact easy to see that there is a commutative diagram
\[\xymatrix{{\bigwedge}^2 T_\IC M \ar[r]^{\hat{\omega}} \ar[d]_{\omega_\IC} & (E, I) \ar[d]^\xi\\
T_\IC M \ar[r]^{\pi^{1, 0}} & E^{1, 0},}\]
where $\xi$ denote the standard complex isomorphism $(E, I) \simeq E^{1, 0}$ as before.

\subsection{Connections on the tangent bundle}

We now want to restrict to the case where $(M, J)$ is an almost complex manifold. Let $\nabla$ be a complex connection on $T^{1, 0} M$: our aim is to give a `good' definition for the torsion of $\nabla$.

Let $D$ be the real connection on $TM$ associated to $\nabla$, which is explicitly given by $D_X Y = \xi^{-1}(\nabla_X \xi(Y))$ and satisfies $DI = 0$. The \emph{holomorphic torsion} of $\nabla$ is then defined as $T^\nabla = \hat{T^D}$, i.e.
\[\begin{array}{rccl}
T^\nabla: & \Gamma\pa{M, {\bigwedge}^2 T_\IC M} & \longrightarrow & \Gamma(M, T^{1, 0} M)\\
 & (X, Y) & \longmapsto & \pi^{1, 0}(D^\IC_X Y - D^\IC_Y X - [X, Y]).
\end{array}\]

\subsection{The case of almost Hermitian manifolds}\label{sect: almost hermitian}

Let $(M, g, J)$ be an almost Hermitian manifold, i.e., $(M, J)$ is an almost complex manifold and $g$ is a Riemannian metric on $M$ such that $g(J \blank, J \blank) = g(\blank, \blank)$. Let $\omega(\blank, \blank) = g(J \blank, \blank)$ be the associated fundamental $2$-form. Then
\[h = g - \ii \omega\]
defines a Hermitian scalar product on $(TM, J)$. Moreover, if we denote by $g_\IC$ the complex bilinear extension of $g$ to $T_\IC M$, then for all $X, Y \in \Gamma(M, TM)$
\[h(X, Y) = 2 g_\IC(\xi(X), \overline{\xi(Y)}),\]
i.e., $\frac{1}{2} h$ coincides with the complex Hermitian extension of $g$ via the canonical identification $(TM, J) \simeq T^{1, 0} M$ provided by $\xi$.

Let now $D$ be a real connection on $TM$, and assume that
\[D g = 0, \qquad DJ = 0.\]
An easy computation then shows that $D \omega = 0$, from which we deduce that $\nabla^D h = 0$.

\begin{rem}\label{rem: g I parallel}
We show now that there exists at least one such connection. Let $D$ be a connection such that $D g = 0$, e.g., the Levi-Civita connection of $g$. Let $D'$ be another connection such that $D' g = 0$: we have $D'_X Y = D_X Y + F_X Y$, and the condition on the metric is equivalent to
\[g(F_X Y, Z) + g(Y, F_X Z) = 0.\]
We want to find a suitable $F$ such that $D' J = 0$. For this purpose, we see that $D' J = 0$ is equivalent to
\[(DJ)_X Y = J F_X Y - F_X JY.\]
So, if we choose
\[F_X Y = -\frac{1}{2} D_X Y - \frac{1}{2} J D_X JY\]
the resulting connection
\[D'_X Y = \frac{1}{2} (D_X Y - J D_X JY)\]
is such that $D' g = 0$ and $D' J = 0$.
\end{rem}

Let $\LC$ denote the Levi-Civita connection of $g$, and consider the connection
\begin{equation}\label{eq: ind conn}
D_X Y = \frac{1}{2} \pa{\LC_X Y - J \LC_X JY}, \qquad X, Y \in \Gamma\pa{M, TM}
\end{equation}
on $TM$. It follows from the discussion in Remark \ref{rem: g I parallel} that $Dg = 0$ and $DJ = 0$, and as a consequence we have the induced (isomorphic) complex connections $\nabla^D$ and $D^\IC$ on $(TM, I)$ and $T^{1, 0} M$ respectively.

We want to compute the holomorphic torsion of these connections, so we begin with some remarks on the torsion of $D$.

\begin{defin}
Let $J$ be an almost complex structure on the differentiable manifold $M$. The \emph{Nijenhuis tensor} of $J$ is
\[N_J(X, Y) = [JX, JY] - J[JX, Y] - J[X, JY] - [X, Y], \qquad X, Y \in \Gamma(M, TM).\]
So $N_J \in \Gamma\pa{M, {\bigwedge}^2 T^*M \tensor TM}$.
\end{defin}

\begin{lemma}
Let $(M, g, J)$ be an almost Hermitian manifold. Denote by $\LC$ the Levi-Civita connection of $g$ and by $D$ the induced connection as in \eqref{eq: ind conn}. Then
\[2 T^D(X, Y) = N_J(X, Y) - (\LC J)_{JX}Y + (\LC J)_{JY}X,\]
where $N_J$ is the Nijenhuis tensor of $J$.
\end{lemma}
\begin{proof}
This is just a computation. On one hand we have
\begin{equation}\label{eq: computation 2T}
\begin{array}{rl}
2 T^D(X, Y) = & \LC_X Y - J \LC_X JY - \LC_Y X + J \LC_Y JX - 2[X, Y] =\\
= & - J \LC_X JY + J \LC_Y JX - \LC_X Y + \LC_Y X =\\
= & J(-\LC_X JY + J \LC_X Y + \LC_Y JX - J \LC_Y X) =\\
= & -J((\LC J)_X Y - (\LC J)_Y X);
\end{array}
\end{equation}
on the other
\begin{equation}\label{eq: computation N_I}
\begin{array}{rl}
N_J(X, Y) = & \LC_{JX} JY - \LC_{JY} JX - J(\LC_{JX} Y - \LC_Y JX) +\\
 & -J(\LC_X JY - \LC_{JY} X) - \LC_X Y + \LC_Y X =\\
= & (\LC J)_{JX} Y - (\LC J)_{JY} X + J(-(\LC J)_X Y + (\LC J)_Y X),
\end{array}
\end{equation}
and the Lemma follows.
\end{proof}

\begin{defin}
Let $(M, g, J)$ be an almost Hermitian manifold, with associated fundamental form $\omega$. Then $(M, g, J)$ is said
\begin{enumerate}
\item \emph{almost K\"ahler} if $d\omega = 0$;
\item \emph{quasi K\"ahler} if $\delbar \omega = 0$.
\end{enumerate}
In particular, any almost K\"ahler manifold is quasi K\"ahler.
\end{defin}

\begin{cor}\label{cor: T = 1/4 N}
Let $(M, g, J)$ be a quasi K\"ahler manifold, and let $\LC$ denote the Levi-Civita connection of $g$. Then $N_J(X, Y) = -2J((\LC J)_X Y - (\LC J)_Y X)$, and so
\[T^D = \frac{1}{4} N_J(X, Y),\]
where $D$ is the connection defined by \eqref{eq: ind conn}.
\end{cor}
\begin{proof}
It follows from \cite[Proposition 1(iv)]{Gau} that $(M, g, J)$ is quasi K\"ahler if and only if $(\LC J)_{JX} Y = -J (\LC J)_X Y$. But then the equation \eqref{eq: computation N_I} semplifies to $N_J(X, Y) = -2J((\LC J)_X Y - (\LC J)_Y X)$. The result then follows from equation \eqref{eq: computation 2T}.
\end{proof}

Under the assumptions of Corollary \ref{cor: T = 1/4 N}, we can see that $T^D$ is of pure type $(0, 2)$: this follows from the fact that the Nijenhuis tensor satisfies $N_J(JX, Y) = -J N_J(X, Y)$. We give now the complex version of the previous result.

\begin{prop}\label{prop: T = N}
Let $(M, g, J)$ be a quasi K\"ahler manifold. Denote by $\LC$ the Levi-Civita connection of $g$ and by $D$ the connection on $TM$ induced by \eqref{eq: ind conn}. Let $\nabla$ be the complex connection on $T^{1, 0} M$ induced by $D$. Then the holomorphic torsion of $\nabla$ is
\[T^\nabla(X, Y) = \frac{1}{4} \pi^{1, 0} N_I^\IC(X, Y), \qquad X, Y \in \Gamma(M, T_\IC M).\]
\end{prop}

\begin{rem}\label{rem: type 02}
We can simplify the expression for $T^\nabla$ further. It is in fact easy to see that for $X, Y \in \Gamma(M, T_\IC M)$ one has
\[N_J^\IC(X, Y) = -4 \pi^{1, 0}[\pi^{0, 1}X, \pi^{0, 1}Y] - 4 \pi^{0, 1}[\pi^{1, 0}X, \pi^{1, 0}Y],\]
and as a consequence
\[T^\nabla(X, Y) = \frac{1}{4} \pi^{1, 0} N_J^\IC(X, Y) = -\pi^{1, 0}[\pi^{0, 1}X, \pi^{0, 1}Y].\]
We can also observe that now it is evident that $T^\nabla$ is a $(0, 2)$-form with values in $T^{1, 0} M$.
\end{rem}

The connection $\nabla$ we defined is the connection appearing in \cite[$\S$2.6]{Gau} corresponding to the parameter $t = 0$. It is uniquely determined by the following conditions:
\begin{enumerate}
\item $\nabla h = 0$;
\item $T^\nabla$ has vanishing $(2, 0)$-part and its $(1, 1)$-part is anti-symmetric.
\end{enumerate}
On the contrary, the \emph{canonical connection} (which is the Chern connection if $I$ is integrable) corresponds to the choice $t = 1$ of the parameter in Gauduchon's paper, and it is characterized by the vanishing of the $(1, 1)$-part of its holomorphic torsion. What Proposition \ref{prop: T = N} and Remark \ref{rem: type 02} show is that, in the case of almost K\"ahler manifolds, these two connections actually coincide.

\begin{notation}
Let $(M, g, J)$ be a quasi K\"ahler manifold, and let $\LC$ be the Levi-Civita connection of $g$. We will denote by $\ccon$ the induced \emph{canonical connection} on $T^{1, 0} M$, i.e., the complex connection
\[\ccon_X Y = \frac{1}{2} \pa{\LC_X Y - J \LC_X JY}, \qquad X \in \Gamma\pa{M, T_\IC M}, Y \in \Gamma\pa{M, T^{1, 0} M}.\]
\end{notation}

\subsection{The complex formalism}

As we are dealing with almost complex manifolds, it is more convenient to work within the \emph{complex} framework, rather than stay with the real formalism.

Let $X = (M, g, J)$ be an almost Hermitian manifold, and let $h$ be the Hermitian scalar product induced by $g$ on $T^{1, 0} M$, namely $h(Z, W) = g_\IC(Z, \bar{W})$ where $Z, W \in \Gamma(M, T^{1, 0} M)$ and $g_\IC$ is the complex bilinear extension of $g$. Fix a (local) $h$-unitary frame $\set{e_1, \ldots, e_n}$ for $T^{1, 0} M$ with dual frame $\set{e^1, \ldots, e^n}$.

Let $\nabla$ be a connection on $T M$ such that $\nabla g = \nabla J = 0$, and denote also by $\nabla$ its extension to $T_\IC M$. The \emph{connection $1$-forms} of $\nabla$ are then the $1$-forms defined by
\[\nabla e_j = \sum_{i = 1}^n \theta^i_j e_i,\]
and they satisfy $\theta^j_i + \bar{\theta}^i_j = 0$. Let $\tau$ be the holomorphic torsion of $\nabla$, then we have $\tau \in \Gamma\pa{M, {\bigwedge}^2 T^*_\IC M \tensor_\IC T^{1, 0} M}$ and so we can write
\[\tau = \sum_{i = 1}^n \Theta^i \tensor e_i.\]
The $2$-forms $\Theta^i$ appearing in this expression are called the \emph{torsion forms} of $\nabla$, and they are related to the connection form by the \emph{first structure equation}
\begin{equation}\label{eq: first structure equations}
\Theta^i = d e^i + \sum_{j = 1}^n \theta^i_j \wedge e^j, \qquad i = 1, \ldots, n.
\end{equation}
Concerning the curvature, we can also decompose the holomorphic curvature of $\nabla$ as follows:
\[R(X, Y)e_j = \sum_{i = 1}^n \Psi^i_j(X, Y) e_i\]
for suitable $2$-forms $\Psi^i_j$, known the as \emph{curvature forms} of $\nabla$. The \emph{second structure equations}
\begin{equation}\label{eq: second structure equations}
\Psi^i_j = d\theta^i_j + \sum_{k = 1}^n \theta^i_k \wedge \theta^k_j, \qquad i, j = 1, \ldots, n
\end{equation}
provide a direct link between the connection forms and the curvature forms.

We focus now on the case where $\nabla$ is the canonical connection $\ccon$ of $X$. Each curvature form $\Psi^i_j$ can be decomposed according to types into its $(2, 0)$, $(1, 1)$ and $(0, 2)$ parts, and we can then define functions $R^i_{jk\bar{l}}$ by the relation
\[(\Psi^i_j)^{1, 1} = \sum_{k, l = 1}^n R^i_{jk\bar{l}} e^k \wedge \bar{e}^l.\]

\begin{defin}[Ricci and scalar curvature]
The \emph{Ricci curvature} of the canonical connection $\ccon$ of an almost Hermitian manifold $(M, g, J)$ is the tensor
\[\Ric(\ccon) = \sum_{k, l = 1}^n R_{k\bar{l}} e^k \wedge \bar{e}^l, \qquad \text{with } R_{k\bar{l}} = \sum_{i = 1}^n R^i_{ik\bar{l}}.\]
The \emph{scalar curvature} of $\ccon$ is the function
\[\scal(\ccon) = \sum_{k = 1}^n R_{k\bar{k}} = \sum_{i, k = 1}^n R^i_{ik\bar{k}}.\]
\end{defin}

\section{The Kodaira--Thurston manifold}\label{sect: kodaira manifold}

Let us consider the differentiable manifold $M = S^1 \times G$, where $S^1$ is a circle and $G$ is the (left) quotient of the Heisenberg group
\[\set{\left( \begin{array}{ccc}
1 & x & z\\
0 & 1 & y\\
0 & 0 & 1
\end{array} \right) \st x, y, z \in \IR}\]
by its subgroup consisting of matrices with integral entries. Call $t$ a coordinate on $S^1$, then $M$ admits the following global fields of tangent vectors
\[e_1 = \frac{\partial}{\partial t}, \qquad e_2 = \frac{\partial}{\partial x}, \qquad e_3 = \frac{\partial}{\partial y} + x \frac{\partial}{\partial z}, \qquad e_4 = \frac{\partial}{\partial z},\]
whose duals are
\[e^1 = dt, \qquad e^2 = dx, \qquad e^3 = dy, \qquad e^4 = dz - x dy.\]
We recall that the only non-trivial differential of the $e^i$'s is $de^4 = - e^2 \wedge e^3$, as the only non-trivial commutator among the global vector fields given above is easily seen to be $[e_2, e_3] = e_4$.

Once we equip $M$ with the complex structure $J$ defined by
\[J e_1 = e_4, \qquad J e_2 = e_3, \qquad J e_3 = -e_2, \qquad J e_4 = -e_1\]
we obtain a complex manifold, which is known as a \emph{Kodaira surface}. It is well known that $\kappa^J(M) = 0$.

In these notes we want to focus on a different (non-integrable) almost complex structure on the same manifold, which was introduced in \cite[$\S$6.1]{Chen-Zhang}. For any $a \in \IR \smallsetminus \set{0}$, the almost complex structure $J_a$ is defined by
\[J_a e_1 = e_2, \qquad J_a e_2 = -e_1, \qquad J_a e_3 = \frac{1}{a} e_4, \qquad J_a e_4 = -a e_3,\]
and it induces the almost complex structure
\[J_a^* e^1 = -e^2, \qquad J_a^* e^2 = e^1, \qquad J_a^* e^3 = -a e^4, \qquad J_a^* e^4 = \frac{1}{a} e_3\]
on the cotangent bundle $T^* M$. The Kodaira dimension $\kappa^{J_a}(M)$ is known.

\begin{prop}[{cf.~\cite[Proposition 6.1]{Chen-Zhang}}]
Consider the almost complex structure $J_a$ on $M$. Then
\[\kappa^{J_a}(M) = \left\{ \begin{array}{ll}
-\infty & \text{for } a \notin \pi \IQ,\\
0 & \text{for } a \in \pi \IQ.
\end{array} \right.\]
\end{prop}

The $2$-form
\[\omega = e^1 \wedge e^2 + e^3 \wedge e^4\]
is a symplectic form on $M$, which is always compatible with $J_a$, meaning that $\omega(J_a \blank, J_a \blank) = \omega(\blank, \blank)$. In the basis of tangent fields $\set{e_1, \ldots, e_4}$, the symmetric bilinear form $g_a(\blank, \blank) = \omega(\blank, J_a \blank)$ is represented by the matrix
\[\left( \begin{array}{cccc}
1 & 0 & 0 & 0\\
0 & 1 & 0 & 0\\
0 & 0 & \frac{1}{a} & 0\\
0 & 0 & 0 & a
\end{array} \right),\]
hence $g_a$ is a Riemannian metric on $M$ if and only if $a > 0$. So from now on we will restrict to the case $a > 0$: we have then an almost K\"ahler manifold $X_a = (M, g_a, J_a)$. We also see that if we let
\[E_1 = e_1, \qquad E_2 = e_2, \qquad E_3 = \sqrt{a} e_3, \qquad E_4 = \frac{1}{\sqrt{a}} e_4,\]
then $\set{E_1, E_2, E_3, E_4}$ is an orthonormal global frame on $X$. Its dual frame is
\[E^1 = e^1, \qquad E^2 = e^2, \qquad E^3 = \frac{1}{\sqrt{a}} e^3, \qquad E^4 = \sqrt{a} e^4,\]
and we easily see that
\[dE^4 = -a E^2 \wedge E^3, \qquad [E_2, E_3] = a E_4.\]

\begin{lemma}\label{lemma: N kodaira surface}
The Nijenhuis tensor $N_{J_a}$ of $X_a$ is given by
\[\begin{array}{lll}
N_{J_a}(E_1, E_2) = 0, & N_{J_a}(E_1, E_3) = a E_3, & N_{J_a}(E_1, E_4) = -a E_4,\\
N_{J_a}(E_2, E_3) = -aE_4, & N_{J_a}(E_2, E_4) = -a E_3, & N_{J_a}(E_3, E_4) = 0.
\end{array}\]
\end{lemma}
\begin{proof}
This is a standard computation. Using the definition, it's easy to see that
\[J_a E_1 = E_2, \qquad J_a E_2 = -E_1, \qquad J_a E_3 = E_4, \qquad J_a E_4 = -E_3,\]
and so $N_{J_a}(E_1, E_3) = a E_3$. The other expressions can be easily deduced from the fact that
\[N_{J_a}(J_a X, Y) = N_{J_a}(X, J_a Y) = -J_a N_{J_a}(X, Y).\]
\end{proof}

Let now $\ccon$ be the canonical connection on $X_a = (M, g_a, J_a)$ introduced in \eqref{eq: ind conn}, and denote by $T = T^{\ccon}$ its torsion. We denote by $\Theta^i$ the real torsion forms of $\ccon$, namely the $2$-forms such that
\[T^{\ccon}(X, Y) = \sum_i \Theta^i(X, Y) E_i.\]

\begin{lemma}
The real torsion forms of the canonical connection $\ccon$ on the almost complex manifold $X_a$ are given by
\[\begin{array}{ll}
\Theta^1 = 0, & \Theta^2 = 0,\\
\Theta^3 = \frac{1}{4} a (E^1 \wedge E^3 - E^2 \wedge E^4), & \Theta^4 = -\frac{1}{4} a (E^2 \wedge E^3 + E^1 \wedge E^4).
\end{array}\]
\end{lemma}
\begin{proof}
By Corollary \ref{cor: T = 1/4 N} we know that $T(X, Y) = \frac{1}{4} N_{J_a}(X, Y)$, hence the result follows from Lemma \ref{lemma: N kodaira surface}.
\end{proof}

We now want to deduce the connection forms of $\ccon$. To set up the notation, we recall that the real connection forms of $\ccon$ are the $1$-forms $\omega^i_j$ such that $\ccon e_j = \sum_i \omega^i_j \tensor e_j$, and we can collect them in the connection matrix $\omega = (\omega^i_j)$ ($i$ is the row index, $j$ is the column index).

\begin{prop}
The real connection matrix for the canonical connection $\ccon$ on the almost complex manifold $X_a$ is
\begin{equation}\label{eq: conn form kodaira surf}
\omega = \frac{1}{4} a \left( \begin{array}{llll}
0 & 0 & E^3 & -E^4\\
0 & 0 & E^4 & E^3\\
-E^3 & -E^4 & 0 & -2 E^2\\
E^4 & -E^3 & 2 E^2 & 0
\end{array} \right).
\end{equation}
\end{prop}
\begin{proof}
We can compute the connection forms $\omega^i_j$ using the Cartan structure equations
\[\left\{ \begin{array}{l}
dE^i + \sum_{j = 1}^4 \omega^i_j \wedge E^j = \Theta^i\\
\omega^i_j + \omega^j_i = 0.
\end{array} \right.\]
In fact, the second set of equations allows us to restrict to $\omega^i_j$ with $j > i$. Hence the first set of equations reduces to
\[\left\{ \begin{array}{l}
\omega^1_2 \wedge E^2 + \omega^1_3 \wedge E^3 + \omega^1_4 \wedge E^4 = 0\\
-\omega^1_2 \wedge E^1 + \omega^2_3 \wedge E^3 + \omega^2_4 \wedge E^4 = 0\\
-\omega^1_3 \wedge E^1 - \omega^2_3 \wedge E^2 + \omega^3_4 \wedge E^4 = \frac{1}{4} a (E^1 \wedge E^3 - E^2 \wedge E^4)\\
-a E^2 \wedge E^3 - \omega^1_4 \wedge E^1 - \omega^2_3 \wedge E^2 - \omega^3_4 \wedge E^3 = -\frac{1}{4} a (E^2 \wedge E^3 + E^1 \wedge E^4),
\end{array} \right.\]
and it is then easy to verify that \eqref{eq: conn form kodaira surf} is the solution of this system.
\end{proof}

From the knowledge of the real connection matrix $\omega$, we can deduce the real curvature matrix $\Omega$ of $\ccon$:
\begin{equation}\label{eq: curv matrix kodaira surf}
\Omega = d\omega + \omega \wedge \omega = \frac{1}{8} a^2 \left( \begin{array}{cccc}
0 & -E^3 \wedge E^4 & E^2 \wedge E^4 & 3 E^2 \wedge E^3\\
E^3 \wedge E^4 & 0 & -3 E^2 \wedge E^3 & E^2 \wedge E^4\\
-E^2 \wedge E^4 & 3 E^2 \wedge E^3 & 0 & E^3 \wedge E^4\\
-3 E^2 \wedge E^3 & -E^2 \wedge E^4 & -E^3 \wedge E^4 & 0
\end{array} \right).
\end{equation}

\begin{thm}
The real scalar curvature of the canonical connection $\ccon$ on the almost complex manifold $X_a$ is given by
\[\scal(D) = -\frac{1}{8} a^2.\]
\end{thm}
\begin{proof}
From the expression of the curvature matrix $\Omega = (\Omega^i_j)$ given in \eqref{eq: curv matrix kodaira surf} we can compute the components $R^i_{jkl}$ of the curvature of $\ccon$, in fact by definition
\[\Omega^i_j = \sum_{k, l} R^i_{jkl} \tensor \pa{E^k \wedge E^l}.\]
As $\set{E_1, \ldots E_4}$ is an orthonormal frame, we have that $R_{ijkl} = R^i_{jkl}$: the non-vanishing components are then
\[\begin{array}{llll}
R_{1234} = -\frac{1}{8} a^2 & R_{1324} = \frac{1}{8} a^2 & R_{1423} = \frac{3}{8} a^2 & R_{2323} = -\frac{3}{8} a^2\\
R_{2424} = \frac{1}{8} a^2 & R_{3434} = \frac{1}{8} a^2 & R_{2134} = \frac{1}{8} a^2 & R_{3124} = -\frac{1}{8} a^2\\
R_{3223} = \frac{3}{8} a^2 & R_{4123} = -\frac{3}{8} a^2 & R_{4224} = -\frac{1}{8} a^2 & R_{4334} = -\frac{1}{8} a^2.
\end{array}\]
As a consequence, the Ricci tensor $R_{ij} = \sum_k R^k_{ikj}$ is expressed by the matrix
\[\left( \begin{array}{cccc}
0 & 0 & 0 & 0\\
0 & 0 & 0 & 0\\
0 & 0 & -\frac{3}{8} a^2 & 0\\
0 & 0 & 0 & \frac{1}{4} a^2
\end{array} \right),\]
and so the scalar curvature is
\[\scal(D) = -\frac{1}{8} a^2.\]
Observe that once again in this last computation we used the fact that $\set{E_1, \ldots, E_4}$ is an orthonormal frame.
\end{proof}

\subsection{An alternative derivation of the connection forms}

Recall from \eqref{eq: ind conn} that the canonical connection is explicitly given by
\[\ccon = \frac{1}{2}(\LC - J_a \LC J_a),\]
where $\LC$ is the Levi-Civita connection of $g_a$. From the Koszul formula expressing the Levi-Civita connection we deduce that in the $g_a$-orthonormal frame of global fields $\set{E_1, E_2, E_3, E_4}$ we have
\[g_a(\LC_{E_i} E_j, E_k) = \frac{1}{2}(g_a([E_i, E_j], E_k) - g_a([E_j, E_k], E_i) - g_a([E_i, E_k], E_j)).\]
In our case, on the almost complex manifold $X_a = (M, g_a, J_a)$ we see that then $g_a(\LC_{E_i} E_j, E_k) = 0$ if both $2$ and $3$ do not appear among $i, j, k$, as the only non-trivial bracket is $[E_2, E_3] = a E_4$. Moreover, for the same reason we see that if $2$ and $3$ appear among $i, j, k$, then $g_a(\LC_{E_i} E_j, E_k)$ is a priori non-trivial only if the remaining index is $4$. Hence the only non-vanishing among the $g_a(\LC_{E_i} E_j, E_k)$ are
\[\begin{array}{lll}
g_a(\LC_{E_2} E_3, E_4) = \frac{1}{2} a & g_a(\LC_{E_2} E_4, E_3) = -\frac{1}{2} a & g_a(\LC_{E_4} E_2, E_3) = -\frac{1}{2} a\\
g_a(\LC_{E_4} E_3, E_2) = \frac{1}{2} a & g_a(\LC_{E_3} E_4, E_2) = \frac{1}{2} a & g_a(\LC_{E_3} E_2, E_4) = -\frac{1}{2} a.
\end{array}\]

We can then use this to compute explicitly how the Levi-Civita connection acts on the basis vector:
\[\begin{array}{llll}
\LC_{E_1} E_1 = 0, & \LC_{E_1} E_2 = 0,                  & \LC_{E_1} E_3 = 0,                 & \LC_{E_1} E_4 = 0,\\
\LC_{E_2} E_1 = 0, & \LC_{E_2} E_2 = 0,                  & \LC_{E_2} E_3 = \frac{1}{2} a E_4, & \LC_{E_2} E_4 = -\frac{1}{2} a E_3,\\
\LC_{E_3} E_1 = 0, & \LC_{E_3} E_2 = -\frac{1}{2} a E_4, & \LC_{E_3} E_3 = 0,                 & \LC_{E_3} E_4 = \frac{1}{2} a E_2,\\
\LC_{E_4} E_1 = 0, & \LC_{E_4} E_2 = -\frac{1}{2} a E_2, & \LC_{E_4} E_3 = \frac{1}{2} a E_2, & \LC_{E_4} E_4 = 0.
\end{array}\]
This result readily implies that
\[\begin{array}{llll}
\ccon_{E_1} E_1 = 0,                  & \ccon_{E_1} E_2 = 0,                  & \ccon_{E_1} E_3 = 0,                 & \ccon_{E_1} E_4 = 0,\\
\ccon_{E_2} E_1 = 0,                  & \ccon_{E_2} E_2 = 0,                  & \ccon_{E_2} E_3 = \frac{1}{2} a E_4, & \ccon_{E_2} E_4 = -\frac{1}{2} a E_3,\\
\ccon_{E_3} E_1 = -\frac{1}{4} a E_3, & \ccon_{E_3} E_2 = -\frac{1}{4} a E_4, & \ccon_{E_3} E_3 = \frac{1}{4} a E_1, & \ccon_{E_3} E_4 = \frac{1}{4} a E_2,\\
\ccon_{E_4} E_1 = \frac{1}{4} a E_4,  & \ccon_{E_4} E_2 = -\frac{1}{4} a E_3, & \ccon_{E_4} E_3 = \frac{1}{4} a E_2, & \ccon_{E_4} E_4 = -\frac{1}{4} a E_1,
\end{array}\]
from which we can compute the connection matrix \eqref{eq: conn form kodaira surf}.

\subsection{Complex curvature of the canonical connection}

It is easy to verify that
\[z_1 = \frac{\sqrt{2}}{2} \pa{E_1 - \ii E_2}, \qquad z_2 = \frac{\sqrt{2}}{2} \pa{E_3 - \ii E_4}\]
is a unitary global frame for $T^{1, 0} M$ with respect to the Hermitian scalar product induced by the complex extension of the metric $g_a$. Its dual frame is given by
\[z^1 = \frac{\sqrt{2}}{2} \pa{E^1 + \ii E^2}, \qquad z^2 = \frac{\sqrt{2}}{2} \pa{E^3 + \ii E^4}.\]

Thanks to the work done in the previous subsections, we can write down the complex connection forms $\theta^i_j$ for the canonical connection $\ccon$:
\[\begin{array}{rl}
\ccon z_1 = & \frac{\sqrt{2}}{2} \pa{\ccon E_1 - \ii \ccon E_2} =\\
= & \frac{\sqrt{2}}{8}a \pa{-E^3 \tensor E_3 + \ii E^4 \tensor E_3 + E^4 \tensor E_4 + \ii E^3 \tensor E_4} =\\
= & \frac{\sqrt{2}}{8}a \pa{(-E^3 + \ii E^4) \tensor (E_3 - \ii E_4)} =\\
= & -\frac{\sqrt{2}}{4}a \bar{z}^2 \tensor z_2;\\
\ccon z_2 = & \frac{\sqrt{2}}{2} \pa{\ccon E_3 - \ii \ccon E_4} =\\
= & \frac{\sqrt{2}}{8}a \left( 2 E^2 \tensor E_4 + E^3 \tensor E_1 + E^4 \tensor E_2 + 2\ii E^2 \tensor E_3 + \right.\\
 & \left. - \ii E^3 \tensor E_2 + \ii E^4 \tensor E_1 \right) =\\
= & \frac{\sqrt{2}}{8}a \left( 2 \ii E^2 \tensor (E_3 - \ii E_4) + \right.\\
 & \left. + (E^3 + \ii E^4) \tensor (E_1 - \ii E_2) \right) =\\
= & \frac{\frac{\sqrt{2}}{4}a z^2 \tensor z_1 \tensor z_2 + \sqrt{2}}{4}a (z^1 - \bar{z}^1).
\end{array}\]
From this computations we deduce that the connection matrix for $\ccon$ is
\[\theta = \frac{\sqrt{2}}{4}a \left( \begin{array}{cc}
0 & z^2\\
-\bar{z}^2 & z^1 - \bar{z}^1
\end{array} \right),\]
hence the curvature matrix is
\[\Psi = d\theta + \theta \wedge \theta = \frac{1}{8}a^2 \left( \begin{array}{cc}
-z^{2\bar{2}} & -2 z^{12} - z^{1\bar{2}} - 2 z^{2\bar{1}} + z^{\bar{1}\bar{2}}\\
-z^{12} - 2 z^{1\bar{2}} - z^{2\bar{1}} + 2 z^{\bar{1}\bar{2}} & z^{2\bar{2}}
\end{array} \right).\]
From this we infer that the only non-trivial coefficient $R^i_{jk\bar{l}}$ are
\begin{equation}\label{eq: coeff}
\begin{array}{lll}
R^1_{12\bar{2}} = -\frac{1}{8}a^2 & R^1_{11\bar{2}} = -\frac{1}{8}a^2 & R^1_{22\bar{1}} = -\frac{1}{4}a^2\\
R^2_{11\bar{2}} = -\frac{1}{4}a^2 & R^2_{12\bar{1}} = -\frac{1}{8}a^2 & R^2_{22\bar{2}} = \frac{1}{8}a^2,
\end{array}
\end{equation}

\begin{thm}\label{thm: scal kodaira}
The Ricci curvature $\Ric(\ccon)$ on the almost Hermitian manifold $X_a = (M, g_a, J_a)$ vanishes. In particular, the scalar curvature $\scal(\ccon)$ also vanishes.
\end{thm}
\begin{proof}
The Theorem follows directly from \eqref{eq: coeff} and the definitions.
\end{proof}

\begin{rem}
The results in Theorem \ref{thm: scal kodaira} can be compared with \cite[Proposition 7.18]{Li}. Our techniques are however different, and can be used to study the behaviour of Kodaira dimension for the other (non-toral) $4$-dimensional almost complex nilmanifolds. We will come back on this topic in a future paper.
\end{rem}

\begin{rem}
For $a = -1$ it is possible to find a different computation of the Ricci tensor of the canonical connection on $X_a$ in \cite[$\S$4]{Tos-Wein}.
\end{rem}

\section{Kodaira dimension of completely solvable Nakamura manifolds}\label{sect: nakamura threefold}

The Nakamura threefold was introduced in \cite[Case III-(3b), p. 90]{Nak}. In the same paper, Nakamura also describes the Kuranishi family of this manifold, and computes the Kodaira dimension of its members. This example showed that the Hodge numbers $h^{p, q}$, the plurigenera and the Kodaira dimension of a complex manifold are not deformation invariants (see \cite[Theorem 2]{Nak}). In this section we endow the differentiable manifold underlying the Nakamura threefold with a family of non-integrable almost complex structures, and compute the Kodaira dimension of its members in Theorem \ref{thm: kod dim}.

We briefly recall the construction of completely solvable Nakamura manifolds. Let $A \in SL(2, \IZ)$ have two real positive distinct eigenvalues
\[\mu_1 =e^{-\zeta}, \qquad \mu_2 =e^{\zeta}, \qquad \zeta \neq 0.\]
Set
\[\Lambda = \left( \begin{array}{cc}
e^{-\zeta} & 0 \\
0 &  e^{\zeta}
\end{array}\right)\]
and let $P \in M_{2,2}(\IR)$ be such that
\[\Lambda= P A P^{-1}\]
Consider $\Gamma := P \IZ^2 \oplus \ii P \IZ^2$; then $\Gamma$ is a lattice in $\IC^2$. Let $\IT^2_{\IC} = \IC^2 / \Gamma$ be a $2$-dimensional complex torus.

Then the map 
\begin{align*}
& \Phi : \IC^2 \longrightarrow \IC^2 \\[5pt]
& \Phi(z) = \Lambda z, \quad \text{where} \quad z = (z^1, z^2)^t,
\end{align*}
induces a biholomorphism of $\IT^2_{\IC}$ by setting $\tilde{\Phi}([z])= [\Phi(z)]$.

First of all, $\tilde{\Phi}$ is well-defined, since if $z'$ and $z$ are equivalent, i.e., if $z'= z + P(\gamma_1 + \ii \gamma_2)$, with $\gamma_1$, $\gamma_2 \in \IZ^2$, then
\begin{align*}
\Phi(z') & = \Lambda z' = \Lambda z + \Lambda P(\gamma_1 + \ii \gamma_2)\\
         & = \Lambda z + PAP^{-1}P (\gamma_1 + \ii \gamma_2)\\
         & = \Lambda z + PA (\gamma_1 + \ii \gamma_2)\\
         & = \Lambda z + P(\lambda_1 + \ii \lambda_2)\\
         & = \Phi(z) + P(\lambda_1 + \ii \lambda_2) \quad \text{with } \lambda_1, \lambda_2 \in \IZ^2,
\end{align*}
so that  $\Phi(z') \sim \Phi(z)$. Furthermore $\tilde{\Phi}^{-1}([z]) = [\Phi^{-1}(z)]$.

We identify $\IR \times \IC^2$ with $\IR^5$ by $(s, z^1, z^2) \longmapsto (s, y^1, y^2, y^3, y^4)$, where $z^1 = y^1 + \ii y^3$, $z^2 = y^2 + \ii y^4$, and consider
\begin{align*}
& T_1 : \IR^5 \longrightarrow \IR^5\\
& T_1(s, y^1, y^2, y^3, y^4) = (s + \zeta, e^{-\zeta} y^1, e^{\zeta} y^2, e^{-\zeta} y^3, e^{\zeta} y^4),
\end{align*}
then 
\[T_1(s, y^1, y^2, y^3, y^4) = T_1(s, z^1, z^2) = (s + \zeta, \Phi(z^1, z^2)).\]
Hence $T_1$ induces a transformation of $\IR \times \IT^2_{\IC}$, by setting
\[T_1(s, [(z^1, z^2)]) = (s + \zeta, [\Phi(z^1, z^2)]).\]

Define 
\[N := \IS^1 \times \frac{\IR \times \IT^2_{\IC}}{<T_1>}.\]
Then, we obtain a family of compact $6$-dimensional solvmanifold of completely solvable type $N$, called \emph{Nakamura manifolds}.

We give a numerical example. Let 
\[A = \left( \begin{array}{cc}
3 & -1 \\
1 &  0
\end{array} \right),\]
so $A \in SL(2, \mathbb{Z})$. Then $\mu_{1, 2} = \frac{3 \pm \sqrt{5}}{2}$. We set 
\[\mu_1 = \frac{3 - \sqrt{5}}{2} = e^{-\zeta} \quad \text{and} \quad \mu_2 = \frac{3 + \sqrt{5}}{2} = e^{\zeta},\]
i.e., $\zeta = \log \pa{\frac{3+\sqrt{5}}{2}}$. Then
\[P^{-1} = \left(\begin{array}{cc}
\frac{3 - \sqrt{5}}{2} & \frac{3 + \sqrt{5}}{2}\\
1 &  1
\end{array}\right),\]
and
\[P = -\frac{1}{\sqrt{5}} \left(\begin{array}{cc}
1 & -\frac{3 + \sqrt{5}}{2}\\
-1 &  \frac{3 - \sqrt{5}}{2}
\end{array}\right)\]
and the lattice $\Gamma$ is given by
\[\Gamma = \Span_{\IZ} < \left[ \begin{array}{c}
-\frac{\sqrt{5}}{5} \\ \frac{\sqrt{5}}{5}\\ 0 \\ 0
\end{array}\right],
\left[\begin{array}{c}
\frac{5 + 3\sqrt{5}}{10}\\ \frac{5 - 3\sqrt{5}}{10} \\ 0 \\0
\end{array}\right],
\left[\begin{array}{c}
0\\ 0 \\ \frac{-\sqrt{5}}{5}\\ \frac{\sqrt{5}}{5} 
\end{array}\right],
\left[\begin{array}{c}
0 \\  0  \\ \frac{5 + 3\sqrt{5}}{10}\\ \frac{5 - 3\sqrt{5}}{10} 
\end{array}\right] >.\]

Going back to the general setting and by using previous notations, it is straightforward to check that
\begin{equation}\label{coframe-Nakamura}
\left\{ \begin{array}{lll}
E^1 := ds, &
E^2 := dx, &
E^3 := e^{s}dy^1,\\
E^4 := e^{-s}dy^2,&
E^5 := e^{s} dy^3,&
E^6 := e^{-s}dy^4.
\end{array} \right.
\end{equation}
gives rise to a global co-frame on $N$, where $dx$ is the global $1$-form on $\IS^1$. Therefore, with respect to $\set{E^i}_{i \in \set{1, \ldots, 6}}$ the structure equations are the following:
\begin{equation}\label{realstructureequation}
\left\{ \begin{array}{lll}
dE^1 = 0,&
dE^2 = 0,&
dE^3 = E^{13},\\
dE^4 = -E^{14},&
dE^5 = E^{15}, &
dE^6 = -E^{16},
\end{array} \right.
\end{equation}
where as usual $E^{ij} := E^i \wedge E^j$. Set
\begin{equation}
\label{almost complex-Nakamura}
\left\{ \begin{array}{l}
J^* E^1 := -E^2,\\
J^* E^3 := -E^4,\\
J^* E^5 := -E^6,
\end{array} \right.
\end{equation}
then $J^*$ is an almost complex on $T^* N$, inducing an almost complex structure $J$ on $N$. Furthermore, 
\[\left\{ \begin{array}{l}
\Phi^1 := E^1 + \ii E^2,\\
\Phi^2 := E^3 + \ii E^4,\\
\Phi^3 := E^5 + \ii E^6.
\end{array} \right.\]
is a complex co-frame of $(1, 0)$-forms on $Y = (N, J)$; one can compute
\[\left\{ \begin{array}{l}
d\Phi^1 = 0,\\
d\Phi^2 = \frac{1}{2}(\Phi^{1 \bar{2}} + \Phi^{\bar{1} \bar{2}}),\\
d\Phi^3 = \frac{1}{2}(\Phi^{1 \bar{3}} + \Phi^{\bar{1} \bar{3}}),
\end{array} \right.\]
where $\Phi^{i\bar{j}} = \Phi^i \wedge \overline{\Phi^j}$ and so on.

Since $b_1(N) = 2$, $b_2(N) = 5$ (see \cite{deBar-Tom}), we obtain
\[\begin{array}{rl}
H^1_{dR} (N; \IR) \simeq & \Span_\IR < E^1, E^2 > = \Span_\IR < \frac{1}{2}(\Phi^{1} + \Phi^{\bar{1}}), \frac{1}{2\ii}(\Phi^{\bar{1}} - \Phi^{1}) >,\\
H^2_{dR}(N; \IR) \simeq & \Span_\IR< E^{12}, E^{34}, E^{56}, E^{36}, E^{45}>\\
\simeq & \Span_\IR <\ii \Phi ^{1\bar{1}}, \ii \Phi ^{2\bar{2}}, \ii \Phi ^{3\bar{3}},\\
 & \ii(\Phi ^{\bar{2}3} - \Phi ^{2\bar{3}}), \ii(\Phi ^{23}-\Phi^{\bar{2}\bar{3}})>.
\end{array}\]
The dual vector fields are given by
\begin{equation}\label{frame-Nakamura}
\left\{ \begin{array}{lll}
E_1 := \frac{\partial}{\partial s}, &
E_2 :=  \frac{\partial}{\partial x}, &
E_3 := e^{-s} \frac{\partial}{\partial y^1},\\
E_4 := e^{s}\frac{\partial}{\partial y^2},&
E_5 := e^{-s} \frac{\partial}{\partial y^3},&
E_6 := e^{s}\frac{\partial}{\partial y^4}.
\end{array} \right.
\end{equation}
Let $\sigma$ be a section of $\cK_Y$. Then $\sigma = f \Phi^1 \wedge \Phi^2\wedge\Phi^3$, where $f$ is a smooth complex valued function on $N$.
\begin{lemma}
$\overline{\partial} \sigma = 0$ if and only if $f = \const$.
\end{lemma}
\begin{proof}
Let $f = u + iv$, where $u: \IR^6 \to \IR$ and $v: \IR^6 \to \IR$ are smooth and $\Gamma$-periodic. Then, since $\overline{\partial} \pa{\Phi^1 \wedge \Phi^2 \wedge \Phi^3} = 0$, we have that $\overline{\partial}\sigma = 0$ if and only if $\overline{\partial}f = 0$. This turns to be equivalent to the following PDEs system
\[\left\{
\begin{array}{l}
u_s - v_x = 0,\\[5pt]
u_x + v_s = 0,\\[5pt]
e^{-s} u_{y^1} - e^s v_{y^2} = 0,\\[5pt]
e^{s} u_{y^2} + e^{-s} v_{y^1} = 0,\\[5pt]
e^{-s} u_{y^3} - e^s v_{y^4} = 0,\\[5pt]
e^{s} u_{y^4} + e^{-s} v_{y^3} = 0.
\end{array} \right.\]
The first two equations imply that $f = f(y^1, y^2, y^3, y^4)$, since $N$ is compact. The other equations imply that $f = \const$ since they form an elliptic PDE system.
\end{proof}

Therefore,
\[P_1(N, J) = 1.\]
Similar computations give
\[P_m(N, J) = 1.\]
Indeed, it is easy to see by induction that $\delbar((\Phi^1 \wedge \Phi^2 \wedge \Phi^3)^{\otimes k}) = 0$ for every $k \geq 1$, so the condition $\delbar \pa{f \cdot (\Phi^1 \wedge \Phi^2 \wedge \Phi^3)^{\otimes k}}$ is again equivalent to $\delbar f = 0$.

\begin{cor}
Let $N$ be a Nakamura manifold of completely solvable type endowed with the (non-integrable) almost complex structure $J$. Then
\[\kappa^J(N) = 0.\]
\end{cor}

\subsection{Kodaira dimension of a deformation of Nakamura manifolds}\label{sect: kod dim of nakamura def}

In this section we will show that the Kodaira dimension is unstable under almost K\"ahler deformations (cf. \cite[Theorem 2]{Nak}). First of all, the following defines a symplectic structure on $N$
\begin{equation}\label{symplectic-Nakamura}
\omega := E^{12}+ E^{34}+ E^{56},
\end{equation}
and $g(\blank, \blank) = \omega(\blank, J\blank)$ gives rise to an almost K\"ahler structure on $N$.

Let $t = (t_1, t_2, t_3, t_4) \in \IR^4$, $\abs{t}^2 < \varepsilon$ and $L_t \in (\End(T N))$ be the endomorphism given by
\[\left\{ \begin{array}{lll}
L_t(E_1) = -t_1 E_1 - t_2 E_2,&
L_t(E_2) = -t_2 E_1 + t_1 E_2,&
L_t(E_3) = E_3,\\
L_t(E_4) = E_4,&
L_t(E_5) = -t_3 E_5 - t_4 E_6,&
L_t(E_6) = -t_4 E_5 + t_3 E_6.
\end{array} \right.\]
Set 
\[J_t = (I + L_t) J (I + L_t)^{-1}.\] 
Then, by a direct computation, we can show the following

\begin{lemma}\label{almost complex-deformation}
The set $\set{J_t}_{t}$ is a family of $\omega$-compatible almost complex structures on $N$ such that $J_0 = J$. Furthermore, setting
\[\left\{ \begin{array}{l}
\alpha(t):= \frac{2t_2}{t_1^2 + t_2^2 - 1},\\
\beta(t) := \frac{(1 - t_1)^2 + t_2^2}{t_1^2 + t_2^2 - 1},\\
\gamma(t):= -\frac{(1 + t_1)^2 + t_2^2}{t_1^2 + t_2^2 - 1},\\
\delta(t):= \frac{2t_4}{t_3^2 + t_4^2 - 1},\\
\lambda(t) := \frac{(1 - t_3)^2 + t_4^2}{t_3^2 + t_4^2 - 1},\\
\mu(t):= -\frac{(1 + t_3)^2 + t_4^2}{t_3^2 + t_4^2 - 1},
\end{array} \right.\]
a $(1,0)$-coframe for $(N, J_t)$ is given by
\[\left\{ \begin{array}{l}
\Phi^1_{t} = E^1 - \ii(\alpha(t) E^1 + \beta(t) E^2),\\
\Phi^2_{t} = E^3 + \ii E^4,\\
\Phi^3_{t} = E^5 - \ii(\delta(t) E^5 + \lambda(t) E^6).
\end{array} \right.\] 
\end{lemma}

With these notations, we have in fact
\begin{equation}\label{eq: ac structure nakamura def}
J_t = \left( \begin{array}{cccccc}
\alpha(t) & \beta(t) & 0 & 0 & 0 & 0\\
\gamma(t) & -\alpha(t) & 0 & 0 & 0 & 0\\
0 & 0 & 0 & -1 & 0 & 0\\
0 & 0 & 1 & 0 & 0 & 0\\
0 & 0 & 0 & 0 & \delta(t) & \lambda(t)\\
0 & 0 & 0 & 0 & \mu(t) & -\delta(t)
\end{array} \right),
\end{equation}
and we get also the relation $-\alpha^2 - \beta \gamma = -\delta^2 - \lambda \mu = 1$.

\begin{lemma}\label{structure-deformation}
We have the following equalities:
\[\delbar \Phi^1_t = 0, \quad \delbar \Phi^2_t = \frac{1}{2} \Phi_t^{1\overline{2}}, \quad \delbar \Phi^3_t = \frac{1}{2} \left[ (1 - \ii\delta(t)) \Phi_t^{1\overline{3}} - \ii \delta(t) \Phi_t^{\overline{1}3} \right].\]
\end{lemma}

By Lemma \ref{almost complex-deformation}, we easily obtain the dual frame $\set{V_1^t, V_2^t, V_3^t}$ of global $(1,0)$-vector fields on $(N, J_t)$:
\[\left\{ \begin{array}{l}
V_1^t = \frac{1}{2} \left[ (E_1 - \frac{\alpha}{\beta}(t) E_2) + \frac{\ii}{\beta(t)} E_2 \right],\\[3pt]
V_2^t = \frac{1}{2} (E^3 - \ii E_4),\\[3pt]
V_3^t = \frac{1}{2} \left[ (E_5 - \frac{\delta}{\lambda}(t) E_6) + \frac{\ii}{\lambda(t)} E_6 \right].\vspace{1mm}
\end{array} \right.\]
More explicitly,
\begin{equation}\label{deformation-complexvector}
\left\{ \begin{array}{l}
V_1^t = \frac{1}{2} \left[ (\frac{\partial}{\partial s} - \frac{\alpha}{\beta}(t) \frac{\partial}{\partial x}) + \frac{\ii}{\beta(t)} \frac{\partial}{\partial x} \right],\\[3pt]
V_2^t = \frac{1}{2} (e^{-s} \frac{\partial}{\partial y_1} - \ii e^s \frac{\partial}{\partial y_2}),\\[3pt]
V_3^t = \frac{1}{2} \left[ (e^{-s} \frac{\partial}{\partial y_3} - \frac{\delta}{\lambda}(t) e^s \frac{\partial}{\partial y_4}) + \frac{\ii}{\lambda(t)} e^s \frac{\partial}{\partial y_4} \right]. \vspace{1mm}
\end{array} \right.
\end{equation}
Let now $\sigma = f \Phi_{t}^{123}$. Then $\delbar \sigma = 0$ if and only if 
\[\delbar f \wedge \Phi_t^{123} - \ii \delta(t) f \Phi_t^{\bar{1}123} = 0,\]
which turns to be equivalent to the system
\begin{equation}\label{eq: pde system}
\left\{ \begin{array}{l}
\bar{V}_1^t f - \ii \delta(t) f = 0,\\
\bar{V}_2^t f = 0,\\
\bar{V}_3^t f = 0.
\end{array} \right.
\end{equation}
By the second and third equation in \eqref{eq: pde system}, we obtain
\[(V_2^t \bar{V}_2^t + V_3^t \bar{V}_3^t) f = 0.\]
As $V_2^t \bar{V}_2^t + V_3^t \bar{V}_3^t$ is a real operator, setting $f = u + \ii v$, the last complex equation is equivalent to the following two real equations,
\begin{equation}
\left\{ \begin{array}{l}
\pa{V_2^t \bar{V}_2^t + V_3^t \bar{V}_3^t} u = 0\\
\pa{V_2^t \bar{V}_2^t + V_3^t \bar{V}_3^t} v = 0.
\end{array} \right.
\end{equation}
By using \eqref{deformation-complexvector}, a direct computation shows that the second order differential operator $4 (V_2^t \bar{V}_2^t + V_3^t \bar{V}_3^t)$ is given by
\[4 (V_2^t \bar{V}_2^t + V_3^t \bar{V}_3^t) = e^{-2s} \frac{\partial^2}{\partial y_1^2} + e^{2s} \frac{\partial^2}{\partial y_2^2} + e^{-2s} \frac{\partial^2}{\partial y_3^2} + \frac{(1 + \delta^2(t))}{\lambda^2(t)} e^{2s} \frac{\partial^2}{\partial y_4^2} - 2 \frac{\delta(t)}{\lambda(t)} \frac{\partial^2}{\partial y_3 \partial y_4}\]
and one can check that it is elliptic. Consequently, $u = u(s, x)$, $v = v(s, x)$, that is $f = f(s, x)$.

The first equation in \eqref{eq: pde system} is then equivalent to the system
\begin{equation}\label{non-homogeneous-system}
\left\{ \begin{array}{l}
2 \delta u = \frac{\partial v}{\partial s} - \frac{\alpha}{\beta} \frac{\partial v}{\partial x} - \frac{1}{\beta} \frac{\partial u}{\partial x}\\
2 \delta v = \frac{\partial u}{\partial s} - \frac{\alpha}{\beta} \frac{\partial u}{\partial x} + \frac{1}{\beta} \frac{\partial v}{\partial x}.
\end{array} \right.
\end{equation}

To resolve system \eqref{non-homogeneous-system}, we begin by observing that it is equivalent to
\[\left\{ \begin{array}{l}
\frac{1}{\beta} \frac{\partial u}{\partial x} = \frac{\partial v}{\partial s} - \frac{\alpha}{\beta} \frac{\partial v}{\partial x} - 2 \delta u\\
\frac{\partial u}{\partial s} = \alpha \pa{\frac{\partial v}{\partial s} - \frac{\alpha}{\beta} \frac{\partial v}{\partial x} - 2 \delta u} - \frac{1}{\beta} \frac{\partial v}{\partial x} + 2 \delta v,
\end{array} \right.\]
hence to
\[\left\{ \begin{array}{l}
\frac{\partial u}{\partial s} = \alpha \frac{\partial v}{\partial s} + \gamma \frac{\partial v}{\partial x} - 2 \alpha \delta u + 2 \delta v\\
\frac{\partial u}{\partial x} = \beta \frac{\partial v}{\partial s} - \alpha \frac{\partial v}{\partial x} - 2 \beta \delta u.
\end{array} \right.\]
Taking the derivative with respect to $x$ of the first equation, and with respect to $s$ of the second one, we can then see that the following relation holds:
\[\pa{\beta \frac{\partial^2}{\partial s^2} - 2 \alpha \frac{\partial^2}{\partial s \partial x} - \gamma \frac{\partial^2}{\partial x^2}} v = 2 \beta \delta \frac{\partial u}{\partial s} - 2 \alpha \delta \frac{\partial u}{\partial x} + 2 \delta \frac{\partial v}{\partial x} = 4 \beta \delta^2 v.\]
Observe that the operator in the left term is elliptic.

So, if $\delta(t) = 0$ (i.e., for $t = (t_1, t_2, t_3, 0)$), we obtain that $v$ must be constant, which forces $u$ to be also constant. This shows that $P_1(M, J_t) = 1$ for $t = (t_1, t_2, t_3, 0)$. We observe that a similar computation shows that $u$ also satisfies
\begin{equation}\label{diff-eq}
\pa{\beta \frac{\partial^2}{\partial s^2} - 2 \alpha \frac{\partial^2}{\partial s \partial x} - \gamma \frac{\partial^2}{\partial x^2}} u = 4 \beta \delta^2 u.
\end{equation}
Since we are looking for periodic solutions of \eqref{diff-eq}, we can work with Fourier series and assume that the solution $u$ is of the form
\[u(s, x) = \sum_{n, m \in \IZ} A_{nm} e^{2\pi \ii \pa{nx + \frac{m}{\zeta} s}}.\]
Assume that $u$ is such a solution, with $A_{nm} \neq 0$ for some pair $(n, m)$: we deduce that the relation
\begin{equation}\label{eq: rel}
\frac{\beta^2}{\zeta^2} m^2 - 2 \frac{\alpha}{\zeta} nm - \gamma n^2 = -\frac{\beta \delta^2}{\pi^2}
\end{equation}
holds, and since $\beta(0) = -1$ we can see this as an equation of degree $2$ in the unknown $m$. The `key observation' is that the discriminant of \eqref{eq: rel}, which is $- \frac{4}{\zeta^2}\pa{n^2 + \frac{\beta^2 \delta^2}{\pi^2}}$, must be non-negative as we are assuming $u$ to be a solution, which forces $n = \beta \delta = 0$. As $\beta(0) = -1$ (and so $\beta(t) \neq 0$ for $\abs{t} < \varepsilon$), the last relation reduces to $\delta = 0$. In particular, this shows that if $\delta(t) \neq 0$, then the only solution to \eqref{diff-eq} and to \eqref{non-homogeneous-system} is the trivial one. Assuming instead that $n = \delta = 0$, relation \eqref{eq: rel} implies that $m = 0$: this means that a non-trivial solution for \eqref{diff-eq} must be constant.

We have then shown that
\[P_1(N, J) = \left\{ \begin{array}{ll}
1 & \text{if } \delta(t) = 0\\
0 & \text{if } \delta(t) \neq 0.
\end{array} \right.\]

Finally, one can prove by induction that $\delbar\pa{(\Phi_t^{123})^{\otimes m}} = -m \frac{\ii \delta}{2} \Phi_t^{\bar{1}} \wedge (\Phi_t^{123})^{\otimes m}$, hence it follows that $\delbar(f (\Phi_t^{123})^{\otimes m}) = 0$ if and only if
\[\bar{V}_1^t f - m \frac{\ii \delta}{2} f = 0, \qquad \bar{V}_2^t f = \bar{V}_3^t f = 0,\]
so the same methods apply also for pluricanonical differentials (just replace $\delta$ with $m \delta$).

\begin{thm}\label{thm: kod dim}
Let $Y_t = (N, g_t, J_t)$ be the almost K\"ahler family of deformations the Nakamura manifold defined above, where $g_t(\blank, \blank) = \omega(\blank, J_t \blank)$. Take any $t = (t_1, t_2, t_3, t_4) \in \IR^4$, $\abs{t}^2 < \varepsilon$. Then
\[\kappa^{J_t}(N) = \left\{ \begin{array}{ll}
0 & \text{if } t_4 = 0,\\
-\infty & \text{if } t_4 \neq 0.
\end{array} \right.\]
\end{thm}

\subsection{Ricci and scalar curvature of the deformed Nakamura manifold}\label{sect: curv of nakamura def}

In this section we consider the almost complex manifolds $(N, J_t)$ where $J_t$ is given by \eqref{eq: ac structure nakamura def}. With the notations introduced in \eqref{coframe-Nakamura}, let us consider the real $(1, 1)$-form
\[\omega = E^1 \wedge E^2 + E^3 \wedge E^4 + E^5 \wedge E^6:\]
it is then easy to observe that
\[d\omega = 0, \qquad \omega(J_t \blank, J_t \blank) = \omega(\blank, \blank).\]
As a consequence, we can endow $(N, J_t)$ with the structure of an almost K\"ahler manifold once we consider the Riemannian metric $g_t$ given as $g_t(\blank, \blank) = \omega(\blank, J_t \blank)$.

A $g_t$-orthonormal frame for $Y_t = (N, g_t, J_t)$ is then provided by
\[\begin{array}{lll}
E_1' = \frac{1}{\sqrt{\gamma}} E_1, & E_3' = E_3, & E_5' = \frac{1}{\sqrt{\mu}} E_5,\\
E_2' = \frac{\alpha}{\sqrt{\gamma}} E_1 + \sqrt{\gamma} E_2, & E_4' = E_4, & E_6' = \frac{\delta}{\sqrt{\mu}} E_5 + \sqrt{\mu} E_6,
\end{array}\]
and with respect to this frame the almost complex structure $J_t$ takes the standard form:
\[\begin{array}{lll}
J E_1' = E_2' & J E_3' = E_4', & J E_5' = E_6',\\
J E_2' = -E_1', & J E_4' = -E_3', & J E_6' = -E_5'.
\end{array}\]
We can then introduce the following complex frame, which is $h_t$-unitary on $T^{1, 0} N$, where $h_t$ denotes the Hermitian extension of $g_t$ to $T_\IC N$:
\[\begin{array}{ll}
z_1 = \frac{\sqrt{2}}{2} \pa{E_1' - \ii E_2'}, & \bar{z}_1 = \frac{\sqrt{2}}{2} \pa{E_1' + \ii E_2'},\\
z_2 = \frac{\sqrt{2}}{2} \pa{E_3' - \ii E_4'}, & \bar{z}_2 = \frac{\sqrt{2}}{2} \pa{E_3' + \ii E_4'},\\
z_3 = \frac{\sqrt{2}}{2} \pa{E_5' - \ii E_6'}, & \bar{z}_3 = \frac{\sqrt{2}}{2} \pa{E_5' + \ii E_6'}.
\end{array}\]

Dually, we can define the coframe
\[\begin{array}{lll}
{E^1}' = \sqrt{\gamma} E^1 - \frac{\alpha}{\sqrt{\gamma}} E^2, & {E^3}' = E^3, & {E^5}' = \sqrt{\mu} E^5 - \frac{\delta}{\sqrt{\mu}} E^6,\\
{E^2}' = \frac{1}{\sqrt{\gamma}} E^2, & {E^4}' = E^4, & {E^6}' = \frac{1}{\sqrt{\mu}} E^6,
\end{array}\]
to which corresponds the complex coframe
\[\begin{array}{ll}
\Phi^1 = \frac{\sqrt{2}}{2} \pa{{E^1}' + \ii {E^2}'}, & \bar{\Phi}^1 = \frac{\sqrt{2}}{2} \pa{{E^1}' - \ii {E^2}'},\\
\Phi^2 = \frac{\sqrt{2}}{2} \pa{{E^3}' + \ii {E^4}'}, & \bar{\Phi}^2 = \frac{\sqrt{2}}{2} \pa{{E^3}' - \ii {E^4}'},\\
\Phi^3 = \frac{\sqrt{2}}{2} \pa{{E^5}' + \ii {E^6}'}, & \bar{\Phi}^3 = \frac{\sqrt{2}}{2} \pa{{E^5}' - \ii {E^6}'}.
\end{array}\]

\begin{lemma}
The (real) torsion forms for the canonical connection $\ccon$ on $Y_t$ are
\[\begin{array}{l}
\Theta^1 = 0,\\
\Theta^2 = 0,\\
\Theta^3 = \frac{1}{2\sqrt{\gamma}} {E'}^{13} + \frac{\alpha}{2\sqrt{\gamma}} {E'}^{14} + \frac{\alpha}{2\sqrt{\gamma}} {E'}^{23} - \frac{1}{2\sqrt{\gamma}} {E'}^{24},\\
\Theta^4 = \frac{\alpha}{2\sqrt{\gamma}} {E'}^{13} - \frac{1}{2\sqrt{\gamma}} {E'}^{14} - \frac{1}{2\sqrt{\gamma}} {E'}^{23} - \frac{\alpha}{2\sqrt{\gamma}} {E'}^{24},\\
\Theta^5 = \frac{1}{2\sqrt{\gamma}}(1 - \alpha \delta) {E'}^{15} + \frac{1}{2\sqrt{\gamma}}(\alpha + \delta) {E'}^{16} + \frac{1}{2\sqrt{\gamma}}(\alpha + \delta) {E'}^{25} - \frac{1}{2\sqrt{\gamma}}(1 - \alpha \delta) {E'}^{26},\\
\Theta^6 = \frac{1}{2\sqrt{\gamma}}(\alpha + \delta) {E'}^{15} - \frac{1}{2\sqrt{\gamma}}(1 - \alpha \delta) {E'}^{16} - \frac{1}{2\sqrt{\gamma}}(1 - \alpha \delta) {E'}^{25} - \frac{1}{2\sqrt{\gamma}}(\alpha + \delta) {E'}^{26},
\end{array}\]
where ${E'}^{ij}$ stands for ${E^i}' \wedge {E^j}'$ and so on.
\end{lemma}
\begin{proof}
Thanks to Corollary \ref{cor: T = 1/4 N}, we only need to compute the Nijenhuis tensor of $J_t$, which can be done by a direct computation.
\end{proof}

\begin{cor}
The (complex) torsion forms for the holomorphic torsion of the canonical connection $\ccon$ on $Y_t$ are
\[\begin{array}{l}
{\Theta^1}' = 0,\\
{\Theta^2}' = \frac{\sqrt{2}}{2\sqrt{\gamma}}(1 + \ii \alpha) \bar{\Phi}^1 \wedge \bar{\Phi}^2,\\
{\Theta^3}' = \frac{\sqrt{2}}{2\sqrt{\gamma}}(1 + \ii \alpha)(1 + \ii \delta) \bar{\Phi}^1 \wedge \bar{\Phi}^3.
\end{array}\]
\end{cor}
\begin{proof}
From the relation
\[\begin{array}{ll}
E_1' = \frac{\sqrt{2}}{2} \pa{z_1 + \bar{z}_1}, & E_2' = \ii \frac{\sqrt{2}}{2} \pa{z_1 - \bar{z}_1},\\
E_3' = \frac{\sqrt{2}}{2} \pa{z_2 + \bar{z}_2}, & E_4' = \ii \frac{\sqrt{2}}{2} \pa{z_2 - \bar{z}_2},\\
E_5' = \frac{\sqrt{2}}{2} \pa{z_3 + \bar{z}_3}, & E_6' = \ii \frac{\sqrt{2}}{2} \pa{z_3 - \bar{z}_3},
\end{array}\]
we can see that the complexified torsion of the canonical connection satisfies
\[T = \Theta^1 \tensor E_1' + \Theta^2 \tensor E_2' + \ldots = \frac{\sqrt{2}}{2} \pa{\Theta^1 + \ii \Theta^2} \tensor z_1 + \frac{\sqrt{2}}{2} \pa{\Theta^1 - \ii \Theta^2} \tensor \bar{z}_1 + \ldots.\]
So the torsion forms for the holomorphic curvature of the canonical connection are
\[\begin{array}{l}
{\Theta^1}' = \frac{\sqrt{2}}{2} \pa{\Theta^1 + \ii \Theta^2},\\
{\Theta^2}' = \frac{\sqrt{2}}{2} \pa{\Theta^3 + \ii \Theta^4},\\
{\Theta^3}' = \frac{\sqrt{2}}{2} \pa{\Theta^5 + \ii \Theta^6}
\end{array}\]
and we can compute them from the knowledge of the real curvature forms.
\end{proof}

Our next step is to compute the connection forms $\theta^i_j$ for the canonical connection $\ccon$, which can be done by solving explicitly the structure equations
\[\left\{ \begin{array}{l}
d\Phi^1 + \theta^1_1 \wedge \Phi^1 + \theta^1_2 \wedge \Phi^2 + \theta^1_3 \wedge \Phi^3 = {\Theta^1}'\\
d\Phi^2 + \theta^2_1 \wedge \Phi^1 + \theta^2_2 \wedge \Phi^2 + \theta^2_3 \wedge \Phi^3 = {\Theta^2}'\\
d\Phi^3 + \theta^3_1 \wedge \Phi^1 + \theta^3_2 \wedge \Phi^2 + \theta^3_3 \wedge \Phi^3 = {\Theta^3}'\\
\theta^i_j + \bar{\theta}^j_i = 0 \qquad (i, j = 1, 2, 3).
\end{array} \right.\]
This is a standard computation, so we prefer to skip all the details and present only the solution in the following Lemma.

\begin{lemma}
Let $Y_t$ be the family of almost K\"ahler deformations of the Nakamura threefold under consideration. The complex connection forms for the canonical connection $\ccon$ of $Y_t$ are then
\[\begin{array}{ll}
\theta^1_1 = 0, & \theta^1_2 = -\frac{\sqrt{2}}{2\sqrt{\gamma}} (1 + \ii \alpha) \Phi^2,\\
\theta^2_2 = 0, & \theta^1_3 = -\frac{\sqrt{2}}{2\sqrt{\gamma}} (1 + \ii \alpha)(1 - \ii \delta) \Phi^3,\\
\theta^2_3 = 0, & \theta^3_3 = \frac{\sqrt{2}}{2\sqrt{\gamma}}(\alpha + \ii)\delta \Phi^1 - \frac{\sqrt{2}}{2\sqrt{\gamma}}(\alpha - \ii)\delta \bar{\Phi}^1.
\end{array}\]
\end{lemma}

From the knowledge of the connection forms, we can deduce the curvature forms via the second structure equations
\[\Psi = (\Psi^i_j) = d\theta + \theta \wedge \theta, \qquad \text{where } \theta = (\theta^i_j).\]
The result is
\[\begin{array}{l}
\Psi^1_1 = -\frac{1}{2\gamma}(1 + \alpha^2) \Phi^2 \wedge \bar{\Phi}^2 - \frac{1}{2\gamma}(1 + \alpha^2)(1 + \delta^2) \Phi^3 \wedge \bar{\Phi}^3,\\
\Psi^1_2 = -\frac{1}{2\gamma}(1 + \alpha^2) \Phi^1 \wedge \bar{\Phi}^2 - \frac{1}{2\gamma}(1 + \ii \alpha)^2 \bar{\Phi}^1 \wedge \bar{\Phi}^2,\\
\Psi^1_3 = \frac{1}{2\gamma}(1 + \alpha^2)(1 + \delta^2) \Phi^1 \wedge \bar{\Phi}^2 + \frac{1}{2\gamma}(1 + \ii \alpha)^2(1 + \delta^2) \bar{\Phi}^1 \wedge \bar{\Phi}^3,\\
\Psi^2_2 = \frac{1}{2\gamma}(1 + \alpha^2) \Phi^2 \wedge \bar{\Phi}^2,\\
\Psi^2_3 = \frac{1}{2\gamma}(1 + \alpha^2)(1 - \ii \delta) \Phi^3 \wedge \bar{\Phi}^2,\\
\Psi^3_3 = \frac{1}{2\gamma}(1 + \alpha^2)(1 + \delta^2) \Phi^3 \wedge \bar{\Phi}^3,
\end{array}\]
and the other curvature forms are deduced from these thanks to the relation $\Psi^i_j + \bar{\Psi}^j_i = 0$.

Recall that the $k\bar{l}$-component of the Ricci curvature of the canonical connection is expressed by
\[R_{k\bar{l}} = \sum_{i = 1}^3 R^i_{ik\bar{l}}, \qquad \text{with } (\Psi^i_j)^{1, 1} = \sum_{k, l} R^i_{jk\bar{l}} \Phi^i \wedge \bar{Phi}^j.\]
We have then no problems with proving the following result.

\begin{thm}\label{thm: scal nakamura}
Let $Y_t$ be the family of almost K\"ahler deformations of the Nakamura threefold under consideration. For every value of the parameter $t$, the canonical connection $\ccon$ on $Y_t$ is Ricci-flat, and in particular its scalar curvature vanishes.
\end{thm}

\begin{rem}
As it was mentioned in the Introduction, we can see that the family of almost K\"ahler structures on the differentiable manifold underlying the Nakamura threefold we are considering has the following properties:
\begin{enumerate}
\item there are members of this family having Kodaira dimension $0$ and $-\infty$;
\item the canonical connection of all the members has vanishing Ricci curvature.
\end{enumerate}
Such a behaviour in the integrable case was pointed out in by Tosatti in \cite[Example 3.2]{Tosatti}, based on the original work of Nakamura (see \cite{Nak}).
\end{rem}

\bibliographystyle{alpha}
\bibliography{Kodaira_dimension_and_curvature}

\begin{thebibliography}{TWY08}

\bibitem[CZ18]{Chen-Zhang}
Haojie Chen and Weiyi Zhang.
\newblock Kodaira dimensions of almost complex manifolds.
\newblock {\em arXiv preprint arXiv:1808.00885 [math.DG]}, 2018.

\bibitem[dBT96]{deBar-Tian}
Paolo de~Bartolomeis and Gang Tian.
\newblock Stability of complex vector bundles.
\newblock {\em J. Differential Geom.}, {\bf 43}(2):231--275, 1996.

\bibitem[dBT06]{deBar-Tom}
Paolo de~Bartolomeis and Adriano Tomassini.
\newblock On solvable generalized {C}alabi--{Y}au manifolds.
\newblock {\em Ann. Inst. Fourier (Grenoble)}, {\bf 56}(5):1281--1296, 2006.

\bibitem[Gau97]{Gau}
Paul Gauduchon.
\newblock Hermitian connections and {D}irac operators.
\newblock {\em Boll. Un. Mat. Ital. B (7)}, {\bf 11}(2, suppl.):257--288, 1997.

\bibitem[Li10]{Li}
Tian-Jun Li.
\newblock Symplectic {C}alabi-{Y}au surfaces.
\newblock In {\em Handbook of geometric analysis, {N}o. 3}, volume~14 of {\em
  Adv. Lect. Math. (ALM)}, pages 231--356. Int. Press, Somerville, MA, 2010.

\bibitem[Nak75]{Nak}
Iku Nakamura.
\newblock Complex parallelisable manifolds and their small deformations.
\newblock {\em J. Differential Geometry}, {\bf 10}:85--112, 1975.

\bibitem[Tos15]{Tosatti}
Valentino Tosatti.
\newblock Non-{K}\"{a}hler {C}alabi--{Y}au manifolds.
\newblock In {\em Analysis, complex geometry, and mathematical physics: in
  honor of {D}uong {H}. {P}hong}, volume~{\bf 644} of {\em Contemp. Math.},
  pages 261--277. Amer. Math. Soc., Providence, RI, 2015.

\bibitem[TW11]{Tos-Wein}
Valentino Tosatti and Ben Weinkove.
\newblock The {C}alabi-{Y}au equation on the {K}odaira-{T}hurston manifold.
\newblock {\em J. Inst. Math. Jussieu}, 10(2):437--447, 2011.

\bibitem[TWY08]{Tos-Wein-Yau}
Valentino Tosatti, Ben Weinkove, and Shing-Tung Yau.
\newblock Taming symplectic forms and the {C}alabi--{Y}au equation.
\newblock {\em Proc. Lond. Math. Soc. (3)}, {\bf 97}(2):401--424, 2008.

\bibitem[Yan17]{Yang2}
Xiaokui Yang.
\newblock Scalar curvature, kodaira dimension and $\hat{A}$-genus.
\newblock {\em arXiv preprint arXiv:1706.01122 [math.DG]}, 2017.

\bibitem[Yan19]{Yang}
Xiaokui Yang.
\newblock Scalar curvature on compact complex manifolds.
\newblock {\em Trans. Amer. Math. Soc.}, {\bf 371}(3):2073--2087, 2019.

\bibitem[Yau74]{Yau}
Shing-Tung Yau.
\newblock On the curvature of compact {H}ermitian manifolds.
\newblock {\em Invent. Math.}, {\bf 25}:213--239, 1974.

\end{thebibliography}

\end{document}